\documentclass[11pt]{amsart}

\usepackage{graphicx}
\usepackage[usenames]{color}
\usepackage{caption}
\usepackage{subcaption}
\usepackage{wrapfig}

\usepackage{verbatim}
\usepackage{url}

\usepackage[centertags]{amsmath}
\usepackage{amsfonts}
\usepackage{amssymb}
\usepackage{amsthm}
\usepackage{newlfont}
\usepackage{url}

\theoremstyle{plain}
\newtheorem{thm}{Theorem}
\newtheorem{cor}[thm]{Corollary}

\newtheorem{lem}[thm]{Lemma}
\newtheorem{prop}[thm]{Proposition}

\theoremstyle{definition}
\newtheorem{defn}{Definition}
\theoremstyle{remark}

\numberwithin{equation}{subsection}
\numberwithin{thm}{section}
\numberwithin{defn}{section}

\newcommand{\defeq}{\mathrel{\mathop:}=}
\newcommand{\ccirc}{\kern0.5ex\vcenter{\hbox{$\scriptstyle\circ$}}\kern0.5ex}

\def\R{\mathbb R}
\def\Z{\mathbb Z}

\def\D{\mathbb D}

\def\Dinf{\partial ^\infty}

\begin{document}
\title[]{Link Obstruction to Riemannian smoothings of locally CAT(0) 4-manifolds}
{\hfil }
\author[]{Bakul Sathaye}

\begin{abstract} 
We extend the methods of Davis-Januszkiewicz-Lafont to provide a new obstruction to smooth Riemannian metric with non-positive sectional curvature. We construct examples of locally CAT(0) 4-manifolds $M$, whose universal covers satisfy isolated flats condition and contain 2-dimensional flats with the property that $\sqcup\ \Dinf F_i \hookrightarrow \Dinf \tilde M$ are non-trivial links that are not isotopic to any great circle link. Further, all the flats in $\tilde M$ are unknotted at infinity and yet $M$ does not have a Riemannian smoothing.
\end{abstract}

\maketitle

\section{Introduction}
Riemannian manifolds with non-positive sectional curvature have been of interest 
for the rich interplay between their geometric, topological and dynamical properties, and powerful local-global properties like the Cartan-Hadamard theorem.
Gromov defined a notion of non-positive curvature for the larger class of geodesic metric spaces. A geodesic metric space is said to have non-positive curvature if for every point, there is a neighborhood such that the geodesic triangles in the neighborhood are ``no fatter'' than Euclidean triangles. These spaces satisfy results that are analogues of results for non-positively curved Riemannian manifolds.
	
We are interested in understanding the relationship between these two notions of curvature for manifolds. In particular, if $M$ is a closed manifold with a locally CAT(0) metric, we want to know if $M$ can support a Riemannian metric with non-positive sectional curvature.
	
In low dimensions, the class of manifolds that support a locally CAT(0) metric is the same as the class of manifolds that support a Riemannian metric of non-positive curvature. In dimension=2, this follows from the classification of surfaces, and 
in dimension=3, this is a consequence of Thurston's geometrization conjecture, which is now a theorem by Perelman and others. 
	
	For dimensions $\ge 5$, Davis and Januszkiewicz \cite{DJ} constructed examples of locally CAT(0) manifolds that do not support Riemannian metrics of non-positive sectional curvature. In fact, they showed that for each $n \ge 5$, there is a piecewise flat, non-positively curved closed manifold $M^n$ whose universal cover $\tilde M^n$ is not simply connected at infinity. They applied hyperbolization techniques to certain non-PL triangulations of $S^n$ to prove this result. This, in particular, means that the universal cover is not homeomorphic to $\R^n$, and by the Cartan-Hadamard theorem, the manifold cannot have a smooth non-positively curved Riemannian metric. 
 	
Recently, Davis, Januszkiewicz and Lafont \cite{DJL} dealt with the remaining case of dimension 4. Their techniques are different from the dimension $\ge 5$ case, and in fact, their example has universal cover  $\tilde M$ diffeomorphic to $\R^4$. 
To prove this, they construct a ``knottedness'' in the boundary of the manifold which gives an obstruction to non-positively curved Riemannian smoothings. 

The purpose of this article is to extend their method to provide new examples of locally CAT(0) 4-manifolds which do not support a non-positively curved Riemannian metric. We show that linking in the boundary of the manifold is another obstruction for a Riemannian smoothing. 
In particular, we prove
\newpage
\begin{thm}
There exists a $4$-dimensional closed manifold $M$ with the following properties:
\vspace{-1mm}
\begin{enumerate}
\item $M$ supports a locally CAT(0)-metric.
\item The boundary of its universal cover $\tilde M$ is homeomorphic to $S^3$, and $\tilde M$ is diffeomorphic to $\R^4$.
\item The maximal dimension of flats in $\tilde M$ is $2$, and the boundary of every $\Z^2$-periodic $2$-flat is a circle that is unknotted in $\Dinf \tilde M$. 
\item $M$ does not support a Riemannian metric of non-positive sectional curvature.
\end{enumerate}
\end{thm}
\vspace{-1mm}
These examples are indeed different from those in \cite{DJL}. Examples in \cite{DJL} have flats that are wild knots in the boundary, where as the examples given by the above theorem have flats that are all unknotted in the boundary. 
In particular, one can see that the obstruction comes from linking of the flats in the boundary, and not from knottendness. 

\vspace{4mm}

\noindent \textbf{Acknowledgements.}
I am grateful to my advisor, Jean-Fran\c{c}ois Lafont, for introducing me to the question addressed in this paper and for his support and guidance through the process.

\section{Background}


In this section we give a brief outline of the proof of the result by Davis-Januszkiewicz-Lafont \cite{DJL}. 
It states the following.
\begin{thm}\label{DJLthm}
There exists a $4$-dimensional closed manifold M with the following four properties:
\begin{enumerate}
\item $M$ supports a locally CAT(0)-metric.
\item $M$ is smoothable, and $\tilde M$ is diffeomorphic to $\R^4$.
\item $\pi_1(M)$ is not isomorphic to the fundamental group of any Riemannian manifold of non-positive sectional curvature.
\item If $K$ is any locally CAT(0)-manifold, then $M \times K$ is a locally CAT(0)-manifold which does not support any Riemannian metric of non-positive sectional curvature.
\end{enumerate}
\end{thm}

To prove this, they start with a link $L \subset S^3$ and construct a triangulation of $S^3$ of type $L$. Recall that a simplicial complex is \textit{flag} provided every $k$-tuple of pairwise incident vertices spans a $(k-1)$-simplex ($k\ge 3$). We say a cyclically ordered 4-tuple of vertices $(v_1, v_2, v_3, v_4)$ in a simplicial complex forms a $square$ provided each consecutive pair of vertices gives an edge in the complex, but the pairs $(v_1, v_3)$ and $(v_2, v_4)$ do not determine an edge. A simpicial complex is said to have \textit{isolated squares} if each vertex of the complex lies in at most one square. 
A triangulation of $S^3$ of $type$ $L$ is a flag triangulation of $S^3$ with isolated squares, where the collection of squares forms the link $L$ up to isotopy. 
\begin{thm} \label{Sigma}
Let $L \subset S^3$ be a link in the 3-sphere. Then there exists a flag triangulation of $S^3$, with isolated squares, and with type the given link $L$. 
\end{thm}
\noindent The construction of such a triangulation is given in detail in Section 3 of \cite{DJL}.

\vspace{3mm}

Starting with a triangulation $\Sigma$ of $S^3$ of type $k$, where $k$ is a non-trivial knot, a locally CAT(0) 4-manifold is constructed. 

For a simplicial complex $K$ with vertex set $I=\{ 1, 2, \ldots, n\}$, one can construct a cubical subcomplex $P_K$ of $[-1,1]^I$ with the same vertex set and with the property that the link of each vertex, $lk_{P_K}(v)$, in $P_K$ is canonically isomorphic to $K$. Here, by the link of $v$ in $P_K$ we mean the set of unit vectors at $v$ that point into $P_K$. This cubical complex $P_K$ can be constructed in the following way.

Let $\mathcal S(K)$ denote the set of all $J \subset I$ such that $J $ is the set of vertices of some simplex $\sigma$ in $L$. $\mathcal S(K)$ is partially ordered by inclusion. Define $P_K$ to be the union of all faces of $[-1,1]^I$ that are a translate of $[-1,1]^J$ for some $J \in \mathcal S(K)$. Such a face will be called a \textit{face of type $J$}. The poset of cells of $P_K$ can be identified with the disjoint union $\underset{J \in \mathcal S(K)}{\coprod} (C_2)^I/(C_2)^J$, where $C_2$ denotes the cyclic group of order 2. $(C_2)^I$ acts simply transitively on the vertex set of $[-1,1]^I$ and transitively on the set of faces of any given type. The stabilizer of a face of type $J$ is the subgroup $(C_2)^J$ generated by $\{r_i\}_{i\in J}$. 
Also recall that one can associate to the 1-skeleton of $K$ a right angled Coxeter group $\Gamma_K$. The associated Davis complex is the universal cover $\tilde P_K$.

\vspace{2mm}
Let $P_\Sigma$ be a cube complex as described above for the triangulation $\Sigma$ of type $k$.
There is a natural piecewise Euclidean metric on $P_\Sigma$, obtained by making each $k$-dimensional face in the cubulation of $P_\Sigma$ isometric to $[-1, 1]^k \subset \R^k$. Using properties of Davis complexes in the right angled case, we see that $P_\Sigma$ has the following properties.
\begin{itemize}
\itemsep 2pt
\item Since $\Sigma$ is a smooth triangulation of $S^3$, $P_\Sigma$ is a smooth 4-manifold. 
\item Since $\Sigma$ is a flag complex, the piecewise flat metric on $P_\Sigma$ is locally CAT(0). 
\item The boundary at infinity, $\Dinf \tilde P_\Sigma$, is homeomorphic to $S^{3}$.
\item By construction, links of vertices in $P_\Sigma$ are simplicially isomorphic to $\Sigma$. 
\end{itemize}
Detailed description of the properties of the cube complex $P_\Sigma$ can be found in the book \cite{D}.


\vspace{2mm}
Define $M \defeq P_\Sigma$. From the above discussion it follows that $M$ has properties (1) and (2) of Theorem \ref{DJLthm}. The diffeomorphism of $\tilde M$ to $\R^4$ follows from a result by Stone \cite[Thm 1]{St}.

It remains to show that $M$ cannot have a Riemannian smoothing. 
Suppose there is a Riemannian manifold $M'$ with non-positive sectional curvature such that $\pi_1(M)$ is isomorphic to $\pi_1(M') $. Let $\Gamma \defeq \pi_1(M)$. The idea is to produce a $\Gamma$-equivariant homeomorphism $\phi : \partial ^\infty \tilde M \rightarrow \partial ^\infty \tilde M'$. Such a homeomorphism may not exist in general (examples by Croke and Kleiner \cite{CK}). However, in this case one can produce a homeomorphism using the following result by Hruska and Kleiner \cite[Cor 4.1.3 and Thm 4.1.8]{HK}:
\begin{thm} For a pair $X_1,X_2$ of CAT(0) spaces with geometric $G$-actions, if $X_1$ has isolated flats, then so does $X_2$, and there is a $G$-equivariant homeomorphism between $\partial ^\infty X_1$ and $\partial ^\infty X_2$. 
\end{thm}
To use this theorem one needs to show that either $\tilde M$ or $\tilde M'$ has isolated flats. By another result of Hruska and Kleiner \cite[Thm 1.2.1]{HK}, we have that if a group $G$ acts geometrically on a CAT(0) space $X$, then $X$ has isolated flats if and only if $G$ is relatively hyperbolic with respect to a collection of virtually abelian subgroups of rank $\ge 2$. 


There exists a geometric action of $\Gamma$ on $ \tilde M$ and we have that $\Gamma$ is a finite index subgroup of the right angled Coxeter group $\Gamma_\Sigma$. So it is enough to show that $\Gamma_\Sigma$ is relatively hyperbolic relative to a collection of virtually abelian subgroups of higher rank. By a criterion developed by Caprace \cite{C}, it suffices to show that $\Sigma$ contains no full subcomplex isomorphic to the suspension $\Sigma K$ of a subcomplex $K$ with 3 vertices which is either (a) the disjoint union of 3 points, or (b) the disjoint union of an edge and 1 point.
In both cases $\Sigma K$ will not have isolated squares. But since $\Sigma$ has isolated squares, $\Sigma$ cannot contain a subcomplex isomorphic to $\Sigma K$. This shows that there exists a $\Gamma$-equivariant homeomorphism $\phi: \partial ^\infty \tilde M \rightarrow \partial ^\infty \tilde M'$.

Now observe that $M$ contains a totally geodesic 2-dimensional flat torus $T^2$. To see this, recall that there is one square given by $k$ in the triangulation $\Sigma$. The cubical complex, $P_\Box$, corresponding to this square is given by $\Box \times \Box$ which is homeomorphic to $S^1 \times S^1 \cong T^2$, a flat torus. 
This gives the flat torus in $M$. For any vertex $v \in T^2 \subset M$, the link at $v$ in $P_\Box$ is $k$ (by properties of $P_\Sigma$) which is knotted in $\Sigma$. So $P_\Box$ is locally knotted in $P_\Sigma$, which means that the torus $T^2$ is locally knotted inside the ambient 4-dimensional manifold $M$. 
By lifting to the universal covers, there is a flat $F$ in $\tilde M$ which is locally knotted at lifts of the vertices in $T^2$. In the boundary at infinity, this gives an embedding of $\partial^\infty F \cong S^1$ into $\partial ^\infty \tilde M \cong S^3$. It turns out that the local knottedness propagates to the boundary at infinity, and this embedding defines a nontrivial knot in $\partial ^\infty \tilde M$. This is done by showing that $\pi_1(\Dinf \tilde M \setminus \Dinf F) \not \cong \Z$ (see Section 4 of \cite{DJL} for details).

Since $\partial ^\infty \tilde M \cong S^3 \cong \partial ^\infty \tilde M'$, $ \phi $ gives a homeomorphism from $S^3$ to $S^3$ which takes $\partial ^\infty F \cong S^1$ to a knotted copy of $S^1$ in $\partial^\infty \tilde M'$. 
The totally geodesic torus $T^2 \hookrightarrow M$ gives us a copy of $\Z^2 \cong \pi_1(T^2)$ in $\Gamma \cong \pi_1(M')$. By the flat torus theorem, there exists a $\Z^2$-periodic flat $F' \subset M'$, with the property that $\partial ^\infty F'$ coincides with the limit set of $\Z^2$. Hence, under the $\Gamma$-equivariant homeomorphism $\phi$, the image of $\partial ^\infty F$ is in fact $\partial ^\infty F'$.

Now, let $p$ be any point on the flat $F'$. The geodesic retraction $\rho' : \partial ^\infty \tilde M' \rightarrow T_p\tilde M'$, where $T_p\tilde M'$ is the unit tangent space at $p$, is a homeomorphism which takes the knotted subset $\partial ^\infty F'$ to the unknotted subset $T_pF'$ lying inside $T_p\tilde M'  \cong S^3$. This gives a contradiction to the existence of the Riemannian manifold $M'$.

\section{Local to Global}

In this section we show that if one were to start with a triangulation of $S^3$ of the type unknot, then the boundary of the flat from the Davis complex corresponding to the unknot is  also an unknot. 

\vspace{2mm}
Let $k$ be a knot in $S^3$, and $\Sigma$ be the triangulation of $S^3$ of type $k$. Let $M = P_\Sigma$ be the 4-manifold as described in Section 2. We know that there is a flat $F$ in the universal cover $\tilde M$ such that if $x \in F$ is a vertex in the cubulation of $\tilde M$, then the link of $x$ in $F$, $lk_F(x)$, is a copy of the knot $k$ in $lk_{\tilde M}(x) \cong S^3$. The geodesic retraction map $\rho : \Dinf \tilde M \rightarrow lk_{\tilde M}(x)$ maps $\Dinf F \subset \Dinf \tilde M$ to $ lk_F(x) \cong k$. 

\begin{thm}\label{localglobal} Let $M, F$ and $x$ be as described above. If $lk_F(x)$ is an unknot in $lk_{\tilde M}(x)$, then there is a homeomorphism of pairs $(\Dinf M, \Dinf F) \cong (lk_{\tilde M}(x), lk_F(x))$. In particular, the knot $\Dinf F \subset \Dinf \tilde M \cong S^3$ is trivial.
\end{thm}

In order to prove this theorem, we first prove a few lemmas.

\begin{lem}\label{lem1} Let $M, F$ and $x$ be as described above. If $lk_F(x)$ is an unknot in $lk_{\tilde M}(x)$, then $F \cap S_x(r) \cong S^1$ is an unknot in the 3-sphere $S_x(r)$ for every $r >0$. 
\end{lem}
\begin{proof}
We do this by ``induction'' on the radius of the sphere, $r >0$. 

We know that $lk_{\tilde M} (x)$ is homeomorphic to a sphere of sufficiently small radius around $x$. So there is a $\delta >0$ such that $lk_{\tilde M} (x) \cong S_x(\delta)$. We can further choose $\delta$ to be small enough so that the closed ball $\overline{B_x(\delta)}$ does not contain any vertices from the cubulation of $\tilde M$ other than $x$. Then the pair $(S_x(\delta), F \cap S_x(\delta))$ is homeomorphic to $(lk_{\tilde M} (x), lk_F(x) )$, where $lk_F(x)$ is nothing but a copy of $k$, and hence $F \cap S_x(\delta)$ is unknotted in $S_x(\delta)$. 

Now let $R \ge \delta$ such that $F \cap S_x (r) \hookrightarrow  S_x (r) $ is an unknot for all $0< r \le R$. 
\vspace{3mm}

\noindent \textit{Case 1}: $S_x(R)$ does not contain any vertices in $F \subset \tilde M$.

Let $S_x(r)$ be a sphere of radius $r > R$, such that the open annular region $A(r, R) \defeq B_{x}(r) \setminus \overline{B_{x}(R)}$ does not contain any vertices in $\tilde M$. Such an $r > R$ exists, since otherwise for some $r_0 > R$ there will be infinitely many vertices in $A(r_0, R)$. This is not true since the set of distances of the vertices from $x$ is discrete.

 Consider the function $D : \tilde M^4 \rightarrow \R$ given by $D(y) = d(x, y)$, where $d$ is the path distance defined on $\tilde M$ using the Euclidean distance on each cube in the cubulation. 
The function $D$ restricts to a smooth function on each cube, and we can consider the gradient field of this restriction on each $i$-cell, 
$\nabla (D \bigr\rvert_{e^i}) $.
\begin{figure}[h]
\centering
\includegraphics[scale=0.9]{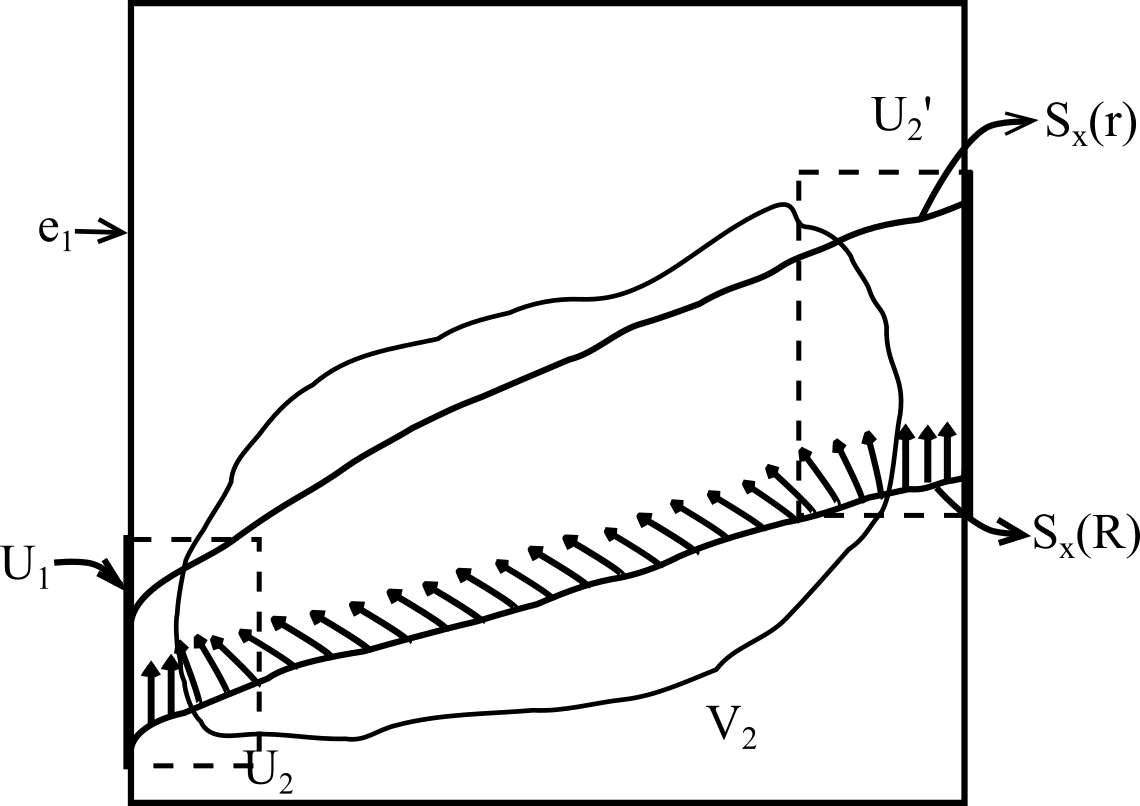}
\caption{A fixed 2-cell $e^2$ in cubulation of $\tilde M$}
\label{localglobal1}
\end{figure}

Let $e^2$ be a fixed 2-cell in the cubulation of $\tilde M$ as shown in Figure \ref{localglobal1}. It is a copy of $[-1,1]^2$ with the Euclidean metric. Suppose $e^1$ is the subcomplex $[-1,1] \times \{ -1 \}$ in $e^2$ that intersects $A(r, R)$, and let $\nabla (D \bigr\rvert_{e^1}) $ be the vector field defined on $e^1$ by the gradient of function $D$. Let $U_1 = ( \alpha, \beta )\times \{ -1 \} \subset e^1$ be an open neighbourhood of $e^1 \cap A(r, R) $ in $e^1$. Define a neighborhood of $e^1\cap A(r,R)$ in $e^2$ as $U_2 \defeq (\alpha , \beta ) \times [-1, \gamma )$, where $-1 < \gamma < -1/2$. Then the metric on $U_2$ is the product metric. 

Using the product structure, one can define a vector field $Y$ on the entire neighbourhood $U_2$ as the pullback of $\nabla (D \bigr\rvert_{e^1}) $ via projection $\pi_1: U_1 \times [0, \alpha) \rightarrow U_1$. 
Similarly, define an open neighbourhood $U_2'$ of $A(r, R) \cap ([-1, 1] \times \{ 1 \})$ as shown and define a vector field $Y'$ on $U_2'$ as above. Together, $Y$ and $Y'$, give a vector field $X_1$ on $U_2 \cup U_2'$.


Now consider an open cover $\{ U_2, U_2', V_2 \}$ of $e^2 \cap A(r, R) $ as shown, and define the vector field, $X_2= \nabla (D \bigr\rvert_{e^2})$ on $V_2$. 
To define a vector field on the open cover, consider a partition of unity for this open cover. There exist functions $f$ and $ g$ such that $f + g = 1$ everywhere, $f$ has compact support inside $U_2 \cup U_2'$ and $g$ has compact support inside $V_2$. Define a new vector field on the open cover by $ f X_1+ g X_2$. This vector field provides a smooth transition from $X_1$ to $X_2$.

Continuing this way, we define a smooth vector field on an open neighbourhood of $A(r, R)  \cap e^i$, for any fixed $i$-cell, $e^i$. Using the fact that $A(r, R)$ intersects finitely many cells we can define vector field on $A(r, R) $, say $X$. Next we normalize this vector field to get a new vector field $\tilde X$ on $A(r, R) $ as follows. 

Let $\phi_y(t)$ be the flow of $X$. We claim that for $y \in S_{x}(R)$, there exists $d_y > 0$ such that $\phi_y(d_y) \in S_{x}(r)$. 
To see this, notice that at every point in $ e^i$, viewing $\nabla (D|_{e^i})$ and $\nabla D$ as vectors in $\R^4$, the inner product $(\nabla (D|_{e^i}) \cdot \nabla D)$ is positive. 
This means that as $t$ increases $D(\phi_y(t))$, the distance of $\phi_y(t)$ from $x$, is strictly increasing. Hence at some $t >0$, $\phi_y(t)$ must reach $S_x(r)$. 



Further, for every $z \in A(r, R) $, there is $y \in S_{x}(R) $ and $t \in [0, d_y]$ such that $\phi_y(t) = z$.
Define, $\tilde X (z) = d_y X(z)$. If $\tilde \phi_y(t)$ denotes the flow of $\tilde X$, then $\tilde \phi_y(1) \in S_{x}(R)$. 
The map $\eta: y \mapsto \tilde \phi_y(1)$ gives a homeomorphism from $S_{x}(r) \rightarrow S_{x}(R)$. 

Further, if $y \in e^i \cap S_x(R)$ for a particular $i$-cell $e^i$, then $\tilde \phi_y(1) \in e^i \cap S_x(r)$. We show this by induction on $i$. It is clear that this is true for $i = 1$. 
Now suppose this is true for $i$ but not for $i +1$, then for some $y \in e^{i+1} \cap S_x(R)$, but $\tilde \phi_y(1) \notin e^{i+1} \cap S_x(r)$. Since $y \mapsto \tilde \phi_y(1)$ is a homeomorphism, for some $e^i \subset \partial e^{i+1}$ and $y' \in S_x(R)$, we must have $\tilde \phi_{y'}(1) \notin e^{i} \cap S_x(r)$ giving a contradiction to our hypothesis. 

This shows that $\eta$ is a homeomorphism $S_x(R) \rightarrow S_x(r)$. Further, since $F$ is a 2-dimensional subcomplex of the cubulation of $\tilde M$, $\eta$ maps $F \cap S_{x}(r)$ to $F \cap S_{x}(R)$ since $F$ . Since $F \cap S_{x}(R)$ is an unknot, $F \cap S_{x}(r)$ must be an unknot in $ S_{x}(r)$.

\begin{figure}[h]
\label{localglobal1a}
\centering
\includegraphics[scale=0.6]{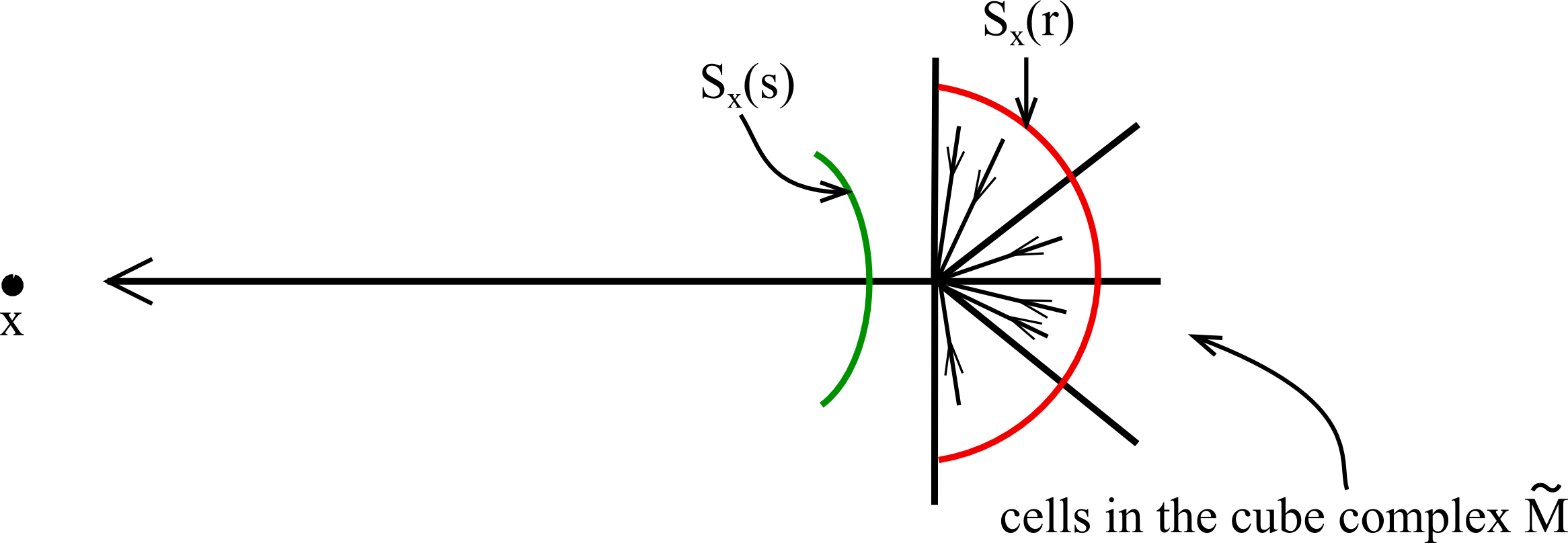}
\caption{}
\label{localglobal1a}
\end{figure}

Note that we restrict the distance function $D$ on each cube, since in general we might encounter singularities in co-dimension 3. A picture in lower dimension is given in Figure \ref{localglobal1a} to indicate such a singularity.

\vspace{3mm}
\noindent \textit{Case 2}: $S_x(R)$ contains a vertex in $F \subset \tilde M$.

Suppose $S_x(R)$ contains one vertex, say $v \in F$. 
First choose an open ball $B_v(\epsilon)$ centered at $v$, with $\epsilon$ small enough so that $F\cap B_v (\epsilon) \cap S_{x}(R)$ is an arc with two endpoints lying on the boundary of $B_v (\epsilon)$. For $R <r< R+ \epsilon$, consider the geodesic retraction map $\rho : S_{x}(r) \rightarrow S_{x}(R)$, and let $B \defeq \rho^{-1}(B_v (\epsilon) \cap S_x(R)) \subset S_{x}(r)$. Define a map $\lambda_1: B \rightarrow S_v(\epsilon )$ such that for $y \in B$, $\lambda_1(y) $ is the point where the geodesic ray from $v$ passing through $y$ intersects the sphere $ S_v(\epsilon )$. See Figure \ref{localglobal2}.

Observe that the restriction of $\rho$ to $B$ is a cell-like map, and hence a near-homeomorphism. So $B$ is homeomorphic to $B_v(\epsilon) \cap S_x(R)$ which is homeomorphic to an open disk. Since $\lambda_1$ is nothing but radial projection of $B$ onto the sphere $S_v(\epsilon )$, $\lambda_1 (B)$ is also homeomorphic to an open disk. 
Hence, the complement of $\lambda_1(B)$ in $S_v(\epsilon)$ is homeomorphic to an open disk. 
On removing this complement, and attaching $\partial \lambda_1(B)$ to $\partial( B_v (\epsilon) \cap S_{x}(R) )$ we get the connect sum $S_x(R) \# S_v(\epsilon)$ and we can define the map 
$$\lambda : S_{x}(r)  \longrightarrow S_{x}(R)\ \#\ S_v(\epsilon)$$
$$y \mapsto \begin{cases} \lambda_1(y),\hspace{2mm} \mathrm{ if}\  y \in B \\
 \rho (y),\hspace{4mm}  \mathrm{if}\ y \notin B &
\end{cases}$$

\begin{figure}[h]
    \includegraphics[width=0.5\textwidth]{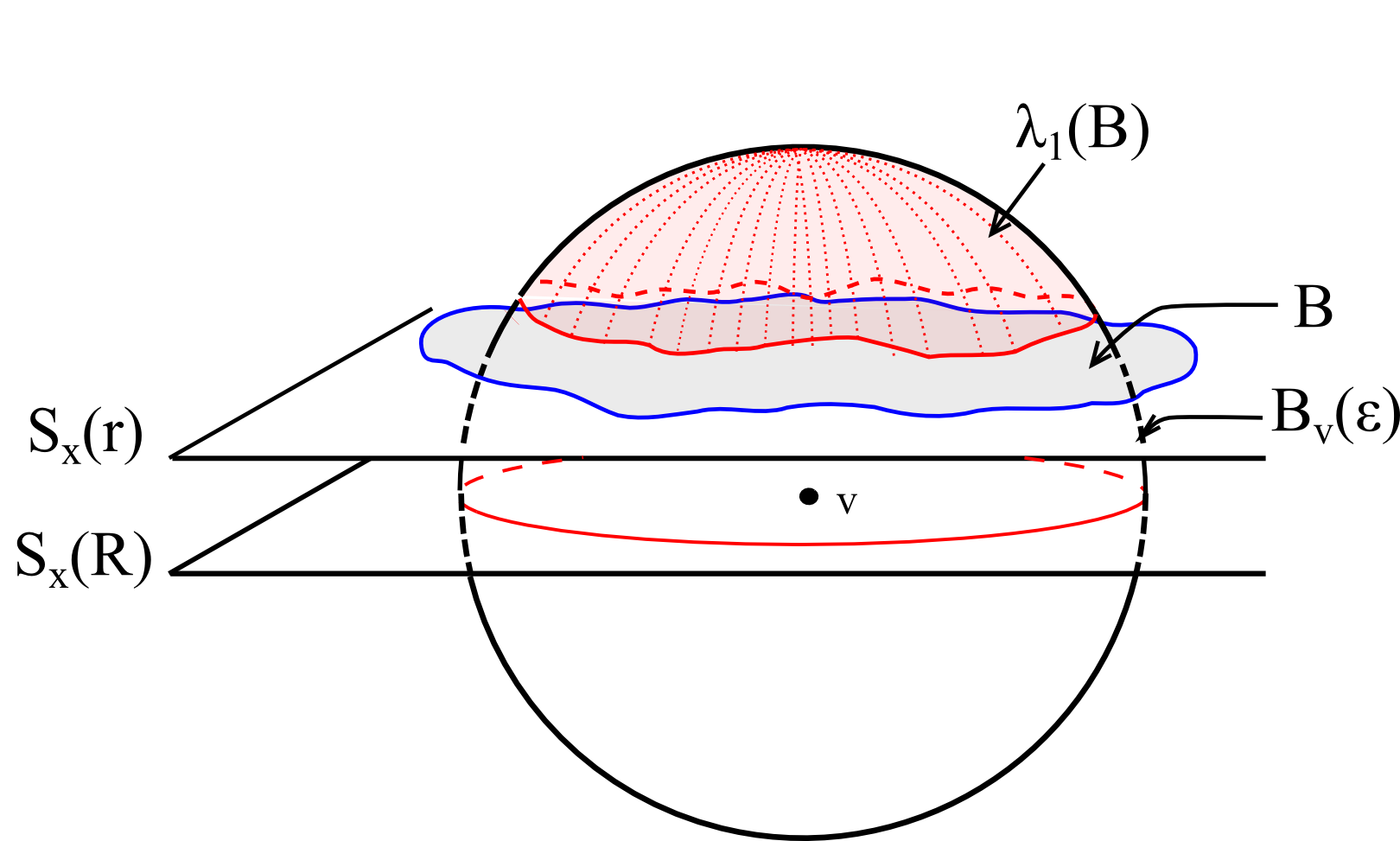}
  \caption{}
  \label{localglobal2}
\end{figure}

Since $\lambda_1$ is the radial projection of $B$ onto the sphere $S_v(\epsilon )$, it maps $B\cap F$ homeomorphically into $S_v(\epsilon) \cap F$. Also, $\rho$ maps the arc $(S_x(r)\setminus B) \cap F$ to the arc $(S_x(R) \setminus B_v (\epsilon) ) \cap F$ homeomorphically. 
To see that $\lambda$ maps $S_x(r) \cap F$ to the connect sum 
 $S_v(\epsilon ) \cap F) \# (S_x(R) \cap F)$, we need to first check that the arcs in the complements, $(S_v(\epsilon) \setminus \lambda_1(B)) \cap F$ and $(S_x(R) \setminus B_v(\epsilon)) \cap F$, are unknotted. 
Observe that $S_x(R) \cap F$ is an unknot by the induction hypothesis. It is well known that unknot cannot be the connect sum of two non-trivial knots. (This can be seen from the fact that genus of a knot is additive \cite{Sc}.) Hence, each of the arcs  $(S_x(R) \cap B_v(\epsilon)) \cap F$ and $(S_x(R) \setminus B_v (\epsilon) ) \cap F$ must unknotted. Also, $ S_v(\epsilon) \cap F$ is homeomorphic to $lk_F(v)$ and which is a copy of the unknot $k$. Hence the arcs $(S_v(\epsilon) \cap \lambda_1(B)) \cap F$ and $(S_v(\epsilon) \setminus \lambda_1(B)) \cap F$ are both unknots. 
Hence $\lambda$ maps $S_x(r) \cap F$ is homeomorphically to the knot connect sum $(S_x(R) \cap F) \# (S_v (\epsilon) \cap F)$. Now it is known that connect sum of two unknots is also an unknot, and hence, $S_x(r) \cap F \hookrightarrow S_x(r)$ must be an unknot. 

If $S_x(R)$ contains more than one vertices, say $ v_1, \ldots, v_k$, then using the above argument we can define a homeomorphism $\lambda : S_x(r) \rightarrow S_x(R) \# S_{v_1}(\epsilon_1) \# \cdots \# S_{v_k}(\epsilon_k)$ which maps $S_x(r) \cap F$ to $(S_x(R) \cap F) \# (S_{v_1}(\epsilon_1) \cap F) \# \cdots \# (S_{v_k}(\epsilon_k) \cap F)$. Since each pair $(S_{v_i}(\epsilon_i) , S_{v_i}(\epsilon_i) \cap F)$ is homeomorphic to $(S^3 , k)$, where $k$ is the unknot, and $(S_x(R) \cap F)$ is an unknot by assumption, the connect sum $(S_x(R) \cap F) \# (S_{v_1}(\epsilon_1) \cap F) \# \cdots \# (S_{v_k}(\epsilon_k) \cap F)$ is also an unknot in $S_x(R) \#S_{v_1}(\epsilon_1) \# \cdots \# S_{v_k}(\epsilon_k)$. 
\end{proof}


Consider the geodesic retraction map $\rho^r_{s} : (S_x(r) , S_x(r) \cap F) \rightarrow (S_x(s) , S_x(s) \cap F)$ as a map of a pair of spaces, where $r>s >0$.

\begin{defn} 
A map $f: X \rightarrow Y$ is said to be a \textit{near-homeomorphism} if it can be approximated arbitrarily closely by homeomorphisms $ X \rightarrow Y$.

A map $f: (X,A) \rightarrow (Y, B)$ is said to be a \textit{near-homeomorphism of pairs} if it can  be approximated arbitrarily closely by homeomorphisms of pairs.  
\end{defn} 

\begin{lem} \label{lem2} Let $M, F$ and $x$ be as before. If $F \cap S_x(r) \hookrightarrow S_x(r)$ is an unknot for every $r >0$, then the geodesic retraction map $\rho^r_{s} : (S_x(r) , S_x(r) \cap F) \rightarrow (S_x(s) , S_x(s) \cap F)$ is a near-homeomorphism of pairs for every $r>s$.
\end{lem}
\begin{proof}
First fix $r > s > 0$. Let $\epsilon >0$ be given. We show that there exists a homeomorphism $\psi :  (S_x(r) , S_x(r) \cap F) \rightarrow (S_x(s) , S_x(s) \cap F)$ such that $|| \psi - \rho_s^r|| < \epsilon.$

Choose $\epsilon_0 < \epsilon /8$. Observe that $\rho = \rho^r_{s}$ can be approximated by diffeomorphisms \cite[Corollary]{M}. Thus, there is a diffeomorphism $\overline \rho : S_x(r) \rightarrow S_x(s)$ such that $||\rho - \overline \rho || < \epsilon_0$. Then $\overline \rho (S_x(r) \cap F) $ is an embedding of $S^1 $ in $S^3$. However, the image $\overline \rho (S_x(r) \cap F)$ might not lie in $S_x(s) \cap F$. So now we construct a homeomorphism $S_x(s) \rightarrow S_x(s)$ such that the image $\overline \rho (S_x(r) \cap F) \subset S_x(s)$ is mapped to $S_x(s) \cap F$.  
 
We know that $S_x(r) \cap F$ is an unknotted piecewise flat embedding of $S^1$ in $S_x(r)$, and hence so is $\overline \rho (S_x(r) \cap F)$.


Choose finitely many points $x_1, \ldots, x_k$ on $\overline \rho (S_x(r) \cap F)$ to give finitely many arcs $\alpha_1, \ldots, \alpha_k$ such that 
\begin{itemize} 
\item[(i)] $x_{i-1}, x_i$ are end points of the arc $\alpha_i$ ($i = 1, \ldots , k$), with $x_0 = x_k$, 
\item[(ii)] $\underset{i}{\cup}\ \alpha_i = \overline \rho (S_x(r) \cap F)$, and
\item[(iii)] $2\epsilon_0 < l_i < 3\epsilon_0$, where $l_i$ is the length of the arc $\alpha_i$ $(i = 1, \ldots, k).$
\end{itemize}

Define $k$ points $y_1, \ldots, y_k$ on $S_x(s) \cap F$ by $y_i \defeq \rho \ccirc \overline \rho^{-1} (x_i)$, and arcs $\beta_i \defeq \rho \ccirc \overline \rho^{-1}(\alpha_i) \ (i = 1, \ldots, k)$. We have the following properties.
 \begin{itemize} 
\vskip-20pt
\item[(i)] $y_{i-1}, y_i$ are end points of the arc $\beta_i$ $(i = 1, \ldots, k)$, with $y_0 = y_k$, 
\item[(ii)] $\underset{i}{\cup}\  \beta_i = S_x(s) \cap F$, and
\item[(iii)] $\beta_i \subset N_{\epsilon_0}(\alpha_i)$, where $ N_{\epsilon_0}(\alpha_i)$ is the $\epsilon_0$-neighborhood of $\alpha_i$.
\end{itemize}


\begin{figure}[h]
\vspace{2mm}
\includegraphics[scale=0.9]{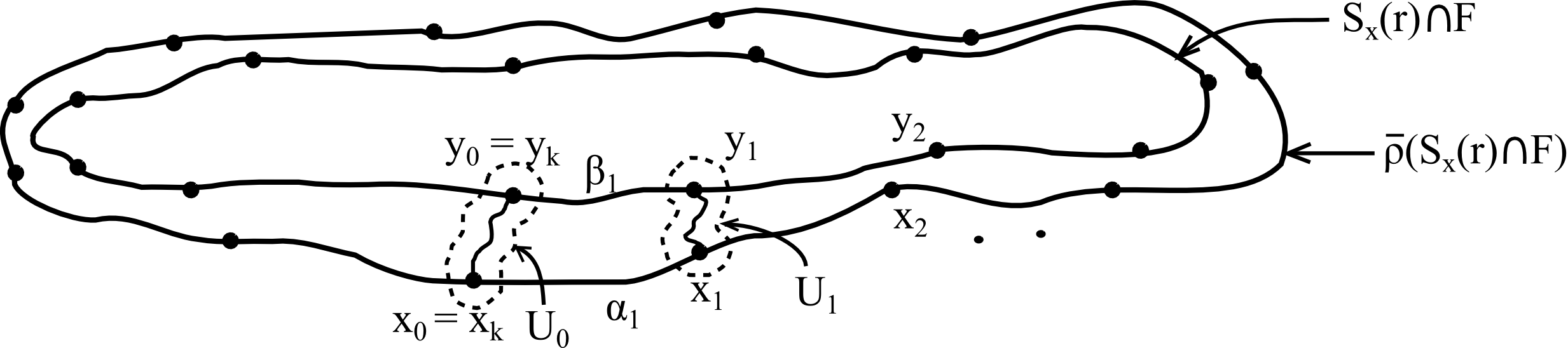}
\caption{}
\label{localglobal3}
\end{figure}

For each $i \in \{ 1, \ldots, k\}$ let $U_i$ be an open neighborhood with diameter $\le 2\epsilon_0$ containing $x_i$ and $y_i$, such that all the $U_i$s are pairwise disjoint and are homeomorphic to open disks. One can construct such neighborhoods as follows. Start with pairwise disjoint paths $\gamma_i$ from $x_i$ to $y_i$ such that $\gamma_i$ are ``unknotted'' in the following sense: there exists a homotopy $G_t : S_x(s) \rightarrow S_x(s)$, $t \in [0,1],$ such that $G_0 (\alpha_i) = \alpha_i$, $G_1 (\alpha_i) = \beta_i$, and $G_t (\gamma_i) \subset \gamma_i$ for every $i \in \{ 1, \ldots , k \}$. Let $\epsilon_1 = \mathrm{min}\{ d(\gamma_i , \gamma_j)\ : 1 \le i \not= j \le k \}$. Choose $\epsilon_2 < \frac{1}{2}\mathrm{min} \{ \epsilon_0, \epsilon_1 \}$ and then define $U_i \defeq N_{\epsilon_2}(\gamma_i)$. See Figure \ref{localglobal3}.

For every $i$, by the Homogeneity lemma \cite{Mil2}, there is a diffeomorphism $f_i: S_x(s) \rightarrow S_x(s)$ that maps $x_i $ to $y_i$ and keeps $S_x(s) \setminus U_i$ fixed.
 Consider the diffeomorphism $f = f_1 \circ f_2 \circ \ldots \circ f_k : S_x(s) \rightarrow S_x(s)$. Then $f \circ \overline \rho (S_x(r) \cap F)$ is a union of flat arcs $f(\alpha_i)$ with endpoints $y_{i-1}$ and $y_i$. Thus, we have pairs of path homotopic flat arcs $f(\alpha_i) $ and $\beta_i$ with same endpoints. Further, since the neighborhoods $U_i$ are disjoint, $f$ is $2\epsilon_0$-close to the identity map on $S_x(s)$. See Figure \ref{localglobal4}.

\begin{figure}[h]
\vspace{3mm}
\includegraphics[scale=0.9]{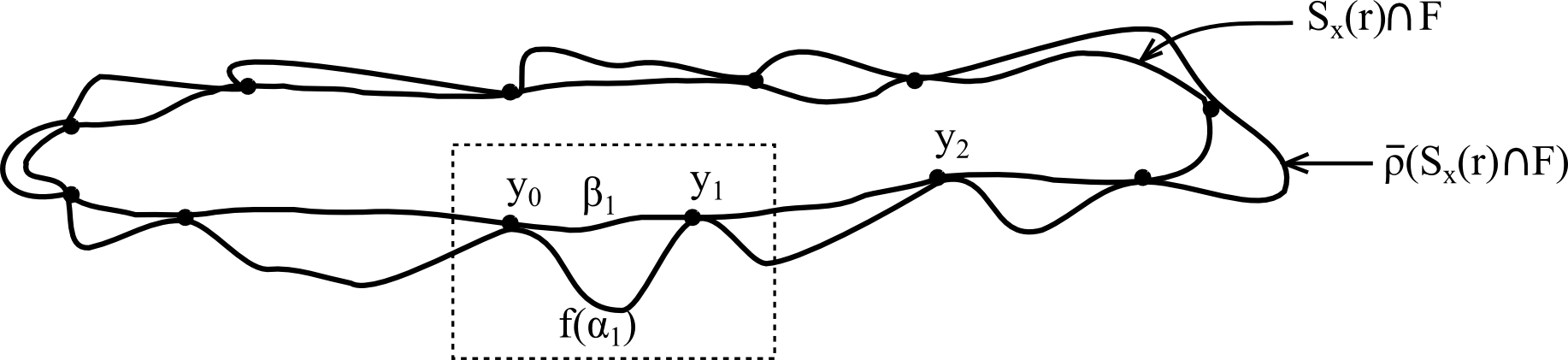}
\caption{}
\label{localglobal4}
\end{figure}

Choose a pair of such arcs, say $f(\alpha_1)$ and $\beta_1$ with the endpoints $y_0$ and $y_1$. 
We now build a homeomorphism $S_x(s) \rightarrow S_x(s)$ that perturbs the arc $f(\alpha_1 )$ so that it coincides with $\beta_1$ in small neighborhoods of $y_0$ and $y_1$. 

Let $K$ be a compact set containing $f(\alpha_1) \cup \beta_1$ such that the endpoints $y_0, y_1$ lie on the boundary $\partial K$, and diam$(K) < 5\epsilon_0$. Such a neighborhood exists, since diam$( N_{\epsilon_0}(\alpha_i)) \le 2\epsilon_0 + l_1 < 5\epsilon_0$.

Since the arcs $f(\alpha_1)$ and $\beta_1$ are piecewise linear, there is a closed ball $B_{y_0}(\delta)$ with $0 < \delta < l_1/3$ such that there is a homeomorphism of $B_{y_0}(\delta)$ onto $\D^3$ which carries the arcs $f(\alpha_1) \cap B_{y_0}(\delta)$ and $\beta_1 \cap B_{y_0}(\delta)$ onto straight line segments in $\D^3$. These arcs intersect the sphere $S_{y_0}(\delta /2)$ in points, say, $S_{y_0}(\delta /2) \cap f(\alpha_1) = \{ a \}$ and $S_{y_0}(\delta /2) \cap \beta_1 = \{ b \}$. Refer to Figure \ref{localglobal5}. 
There exists an ambient isotopy $g$ of $\overline{ B_{y_0}(\delta /2) } \cap K$ that fixes the boundary $\partial (B_{y_0}(\delta /2) \cap K)$ and 
maps $f( \alpha_1)$ to $\beta_1$. In particular, it maps $a$ to $b$. 

\begin{figure}[h]
\begin{subfigure}{0.45\textwidth}
\centering
\includegraphics[scale=0.8]{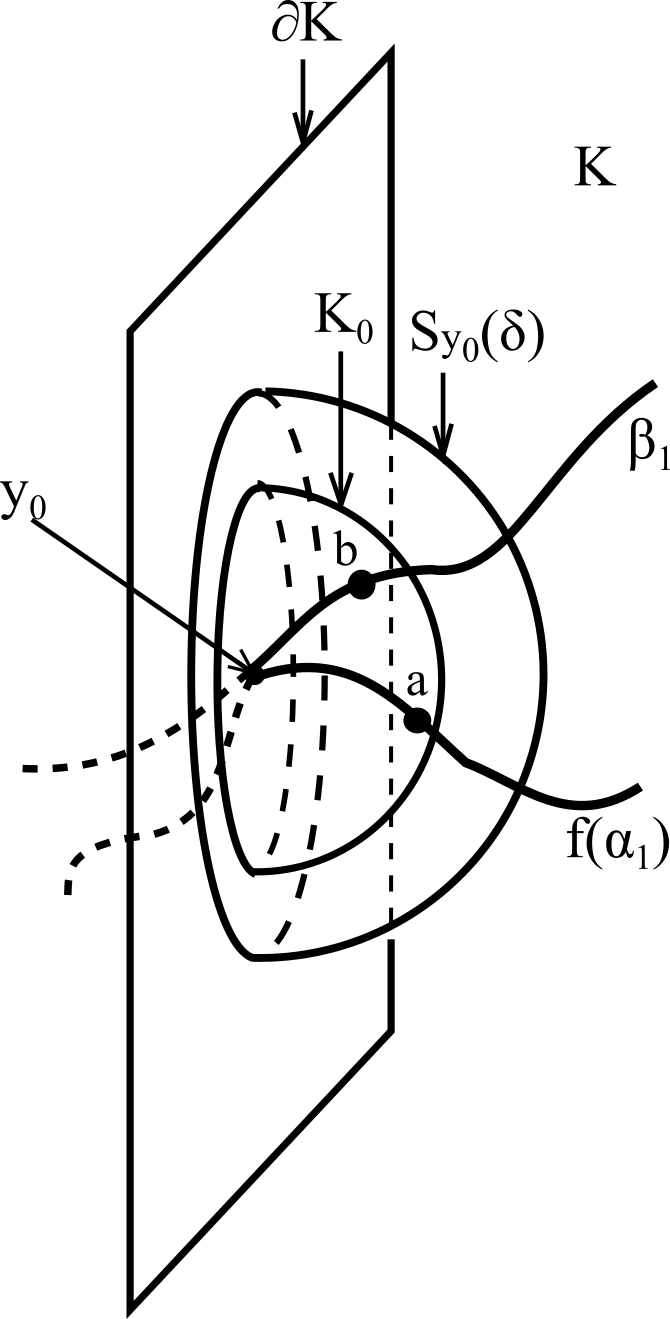}
\caption{}
\label{localglobal5}
\end{subfigure}
\begin{subfigure}{0.45\textwidth}
\centering
\includegraphics[scale=0.8]{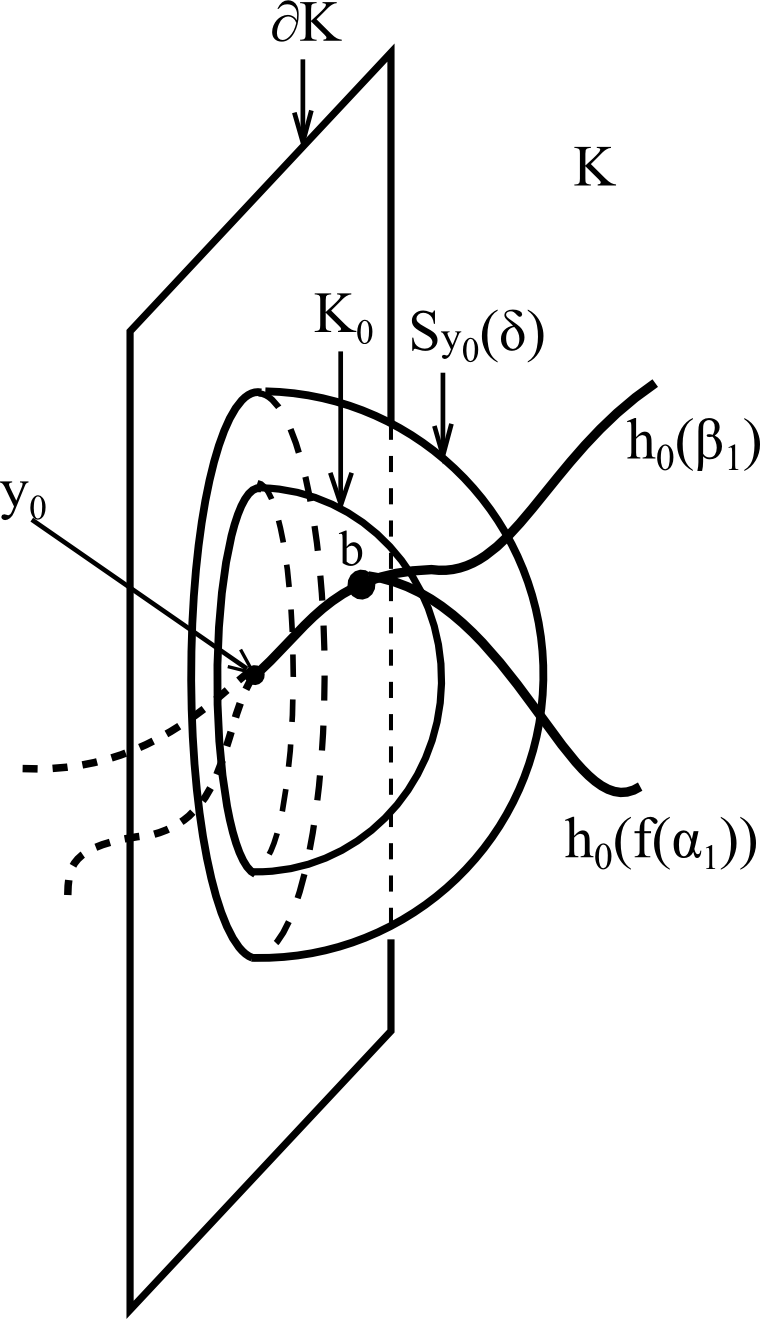}
\caption{}
\end{subfigure}
\caption{}
\end{figure}

Let 
$A_K(\delta, \delta /2) \defeq (\overline{B_{y_0}(\delta)} \setminus B_{y_0}(\delta /2)) \cap K$ be the closed annular region in $K$ between the two spheres. Observe that $K_0 \defeq S_{y_0}(\delta /2) \cap K$ is homeomorphic to $\D^2$ and that $A_K(\delta, \delta /2)$ is homeomorphic to $\D^2 \times I$. Let $\phi : K_0 \rightarrow \D^2$, and $\tilde \phi : A_K(\delta, \delta /2) \rightarrow \D^2 \times I$ be a pair of homeomorphisms such that $\tilde \phi |_{K_0} = \phi$, and $\tilde \phi (A_K(\delta, \delta /2) \cap f(\alpha_1) ) = \{ \phi(a) \} \times I$ and $\tilde \phi (A_K(\delta, \delta /2) \cap \beta_1 )  = \{ \phi (b) \} \times I$.

Now the homeomorphism $g$ restricted to $S_{y_0}(\delta /2)  \cap K$ 
maps $f( \alpha_1)$ to $\beta_1$ and is isotopic to the identity map, which in turn gives a homotopy of the disc $\D^2$, 
$H_t : \D^2  \rightarrow \D^2$ ($t \in [0, 1]$) such that $H_0 = \phi \circ g \circ \phi^{-1}$ and $H_1 = id$. One can use $H_t$ to define a homeomorphism $\tilde H : \D^2 \times I \rightarrow \D^2 \times I$ given by $(x, t ) \mapsto (H_t(x), t )$.

This in turn gives a map $\tilde \phi ^{-1} \circ \tilde H \circ \tilde \phi : A_K(\delta, \delta /2) \rightarrow A_K(\delta, \delta /2)$ that fixes the boundary, $\partial (A_K(\delta, \delta /2))$. This together gives a homeomorphism $h_0 : S_x(s) \rightarrow S_x(s)$ such that $h_0|_{A_K(\delta, \delta /2)} = \tilde \phi ^{-1} \circ \tilde H \circ \tilde \phi $ and $h_0 = id$ outside $A_K(\delta, \delta /2)$. 
The arcs $h_0(f(\alpha_1)) $ and $h_0 (\beta_1)$ coincide in the ball $K \cap B_{y_0} (\delta /2)$. 

Similary, one can define another homeomorphism $h_1 : S_x(s) \rightarrow S_x(s)$ which gives arcs that coincide in a neighbourhood, $K \cap B_{y_1} (\delta' /2)$ of $q$, for some $0 <\delta ' < l_1/3 $. This gives us arcs $h_1 \circ h_0 (f(\alpha_1))$ and $h_1 \circ h_0 (\beta_1)$ that coincide in the neighbourhoods $K \cap B_{y_0} (\delta /2)$ and $K \cap B_{y_1} (\delta' /2)$. Consider the new arcs $\alpha_1 '$ and $\beta_1 '$ with common endpoints $b = y_0''$ and $y_1'$ 
as shown. Observe that they are path homotopic. 

We now apply a result by Martin and Rolfsen \cite{MR} which says that homotopic arcs are, in fact, isotopic. Let $U$ be an open set in $S_x(s)$ such that $\alpha' \cup \beta' \subset U \subset K$. The result in our context states the following:

\begin{thm}[Martin-Rolfsen]
\label{martinrolfsen}
Let $\alpha'$ and $\beta'$ be path homotopic flat arcs in $S_x(s)$ with end points $a'$ and $b'$, and $U$ be an open set containing $\alpha'$ and $\beta'$ as described above. Then there exists an ambient isotopy $\Phi_t (0\le t \le 1)$ of $S_x(s)$ fixed on $a'$ and $b'$ such that $\Phi_1(\alpha') = \beta'$. Further, this isotopy is fixed on $S_x(s)\setminus U$. 
\end{thm}

Now consider the homeomorphism on $S_x(s)$ given by $\phi_1 \defeq \Phi_1\circ h_1 \circ h_0$, where $\Phi_1$ is as in Theorem \ref{martinrolfsen}. Observe that $\phi_1(f(\alpha_1)) = \beta_1$ and it fixes $(S_x(s) \setminus K) \cup \partial K$. Since the diam$(K) < 5\epsilon_0$, $\phi_1$ is $5\epsilon_0$-close to the identity map. 


This process can be done for every pair of arcs $f(\alpha_i)$ and $\beta_i$ to obtain homeomorphisms $\phi_i: S_x(s) \rightarrow S_x(s)$ which map $f(\alpha_i)$ to $\beta_i$, fix $(S_x(s) \setminus K_i) \cup \partial K_i$ ($K_i$ are compact sets as above), and stay $ 5\epsilon_0$-close to $id$.
 This gives a homeomorphism $\phi \defeq \phi_1 \circ \cdots \circ \phi_k$ which maps each $\alpha_i$ to $\beta_i$ and hence maps $\overline \rho (S_x(s) \cap F)$ to $S_x(s) \cap F$. Observe that the interiors of $K_i$ are pairwise disjoint, and hence the compostion of $\phi_i$ is also $5\epsilon_0$-close to $id$. 

Define $\psi = \phi \circ f \circ \overline \rho : S_x(r) \rightarrow S_x(s)$ that maps  $S_x(r) \cap F$ to $S_x(s) \cap F$. Further, $\psi$ is $\epsilon$-close to $\rho$ since
\begin{eqnarray*}
d(\rho (x) , \psi (x) ) & \le & d(\rho (x) , \overline \rho (x) ) + d(\overline \rho (x) , f \circ \overline \rho (x) ) + d(f \circ \overline \rho (x) , \phi \circ f \circ \overline \rho (x) )\\
& < & \epsilon_0 + 2\epsilon_0 + 5\epsilon_0\\
& < & \epsilon.
\end{eqnarray*}
Since $\epsilon > 0$ was arbitrary, this shows that every map $\rho^r_{s} = \rho : (S_x(r), S_x(r) \cap F) \rightarrow (S_x(s), S_x(s) \cap F)$ is indeed a near-homeomorphism of pairs.  
\end{proof}

\vspace{2mm}

Now consider the following result regarding inverse limits of sequences of spaces \cite{B}. 
\begin{thm}[Brown] \label{Brown} 
Let $X = \underset{\longleftarrow}{\lim} (X_n, f_n)$ where the $X_n$ are all homeomorphic to a compact metric space $S$, and for all $n$, $f_n$ is a near homeomorphism. Then $X$ is homeomorphic to $S$. 
\end{thm} 
 
Lemma \ref{thm for pairs} states this result for an inverse system of pairs of spaces.

\begin{lem} \label{thm for pairs} Let $(X, A)= \underset{\longleftarrow}{\lim}((X_i, A_i), f_i)$, where $X_i$ are all homeomorphic to a compact space $S$, and all $A_i$ are homeomorphic to a compact metric space $T$. For $i\ge 2$, let $f_i: (X_i, A_i) \rightarrow (X_{i-1}, A_{i-1})$ be a near homeomorphism of pairs. Then $(X,A)$ is homeomorphic to $(S, T)$.  
\end{lem}

\begin{proof} To prove this, let us consider the following results by M. Brown \cite{B}. 

\begin{prop}[Brown] \label{thm} Let $X = \underset{\longleftarrow}{\lim}(X_i, f_i)$ and $Y =  \underset{\longleftarrow}{\lim}(X_i,g_i)$, where $X_i$ are compact metric spaces. Suppose $||f_{i+1} - g_{i+1}|| < a_i$, $i = 1, 2, \ldots $, where $(a_i)$ is a Lebesgue sequence for $(X_i, g_i)$. Then the function $F_N : X \rightarrow X_N$ defined by $F_N = \underset{n \rightarrow \infty}{\lim}\ g^n_{N}f^\infty_{n}$ is well defined and continuous. Moreover, the function $F : X \rightarrow Y$ defined by $F(s) = (F_1(s), F_2(s), \ldots )$ is well defined, continuous, and onto.  
\end{prop} 
 
A \textit{Lebesgue sequence} for a sequence of spaces $(X_i, g_i)$ is a sequence $(a_i)$ of positive real numbers such that there is another sequence $(b_i)$ of positive real numbers which satisfies: (1) $\sum b_i < \infty$ and (2) whenever $x, y \in X_j, i< j,$ and $|x-y| <a_j$, then $|f^j_{i}(x) - f^j_{i}(y)| <b_j$.  
 
\vspace{2mm}
Most of the proof of Prop. \ref{thm} follows through for pairs $(X_n, A_n)$. One only needs to check that the map $F: \underset{\longleftarrow }{\lim} (A_i, f_i) \rightarrow \underset{\longleftarrow}{\lim} (A_i, g_i)$ is onto. The proof uses the fact that for any $i >0$, $\underset{j=i}{\overset{\infty}{\cap}}\ f^j_{i}(X_j) = f^\infty_{i} (X)$ (from Corollary 3.8 Chapter VIII of \cite{ES}). The same holds true for the sequence of subspaces $(A_i)$, that is  $\underset{j=i}{\overset{\infty}{\cap}}\ f^j_{i}(A_j) = f^\infty_{i} (A)$. Hence, a similar argument as for surjectivity of $F: \underset{\longleftarrow}{\lim} (X_i, f_i) \rightarrow \underset{\longleftarrow}{\lim} (X_i, g_i)$ shows that given $b  \in \underset{\longleftarrow}{\lim} (A_i, g_i)$, there is $a \in A$ such that $F(a) = b$. This gives the following result.
 
\begin{prop} Let $(X, A) = \underset{\longleftarrow}{\lim}(X_i,A_i, f_i)$ and $(Y,B) =  \underset{\longleftarrow}{\lim}(X_i,A_i, g_i)$, where $X_i$ and $A_i$ are compact metric spaces. Suppose $||f_{i+1} - g_{i+1}|| < a_i$, $i = 1, 2, \ldots $, where $(a_i)$ is a Lebesgue sequence for $(X_i, g_i)$. Then the function $F_N : (X,A) \rightarrow (X_N, A_N)$ defined by $F_N = \underset{n \rightarrow \infty}{\lim}\ g^n_{N}f^\infty_{n}$ is well defined and continuous, and further, the function $F : (X,A) \rightarrow (Y,B)$ defined by $F(s) = (F_1(s), F_2(s), \ldots )$ is well defined, continuous, and onto.  
\end{prop} 

 
The following theorem shows that $F$ is in fact a homeomorphism under some extra conditions on the maps $f_n$. 

\begin{thm}[Brown] Let $X= \underset{\longleftarrow}{\lim}(X_i, f_i)$, where $X_i$ are compact metric spaces. For $i\ge 2$ let $K_i$ be a nonempty collection of maps from $X_i $ into $X_{i-1}$. Suppose that for each $i\ge 2$ and $\epsilon >0$ there is $g\in K_i$ such that $||f_i - g|| < \epsilon$. Then there is a sequence $(g_i) $ with $g_i \in K_i$ such that $X $ is homeomorphic to $\underset{\longleftarrow}{\lim}(X_i, g_i). $ 
\end{thm} 

This theorem follows from Prop. \ref{thm}. 
In case of pairs $(X_n, A_n)$, injectivity of $F$ on $A =\underset{\longleftarrow }{\lim} (A_n, f_n) $ follows from the injectivity on the ambient space $X $. Similarly, continuity of the inverse also follows and we get the following result.

\begin{prop} \label{propn} Let $(X,A) = \underset{\longleftarrow}{\lim}(X_i, A_i, f_i)$, where $X_i$ and $A_i$ are compact metric spaces. For $i\ge 2$ let $K_i$ be a nonempty collection of maps from $(X_i, A_i) $ into $(X_{i-1}, A_{i-1})$. Suppose that for each $i\ge 2$ and $\epsilon >0$ there is $g\in K_i$ such that $||f_i - g|| < \epsilon$. Then there is a sequence $(g_i) $ with $g_i \in K_i$ such that $(X, A) $ is homeomorphic to $\underset{\longleftarrow}{\lim}(X_i, A_i, g_i). $ 
\end{prop}

It is not hard to see that Lemma \ref{thm for pairs} follows from Prop. \ref{propn} by taking $K_i$ to be the set of homeomorphisms $(X_i, A_i)  \rightarrow (X_{i_1}, A_{i_1})$.

 \end{proof}

Now we are in a position to prove Theorem \ref{localglobal}. 

\begin{proof}
If $lk_F(x)$ is an unknot in $lk_{\tilde M}(x)$, then by Lemma \ref{lem1}, $F\cap S_x(r) \hookrightarrow S_x(r)$ is an unknot for every $r>0$. By Lemma \ref{lem2}, the geodesic retraction map $\rho^r_{s} : (S_x(r) , S_x(r) \cap F) \rightarrow (S_x(s) , S_x(s) \cap F)$ is a near-homeomorphism of pairs for every pair $r>s$.

Now, we know that $(\Dinf \tilde M, \Dinf F) = \underset{\longleftarrow}{\lim} ((S_x(r) , S_x(r) \cap F), \rho_r)$.
Since $\rho_s^{r}$ is a near-homeomorphism of pairs for every $r>s$, by Lemma \ref{thm for pairs}, we have that $(\Dinf M, \Dinf F)$ is homeomorphic to the pair $(S^3, S^1)$ with $S^1 \hookrightarrow S^3$ being an unknot. 
 
\end{proof}

The following corollary generalizes Theorem \ref{localglobal} for links whose components are all unknots.
\begin{cor}
\label{cor}
Let $L = (l_1, \ldots , l_m)$ be an $m$-component link in $S^3$ and $\Sigma$ be the triangulation of $S^3$ of type $L$. Let $M$ be the manifold as defined in Section 2. Let $\{ F_1, F_2, \ldots, F_m \} $ be the collection of flats such that $lk_{F_i}(x)$ is a copy of the $i$-th component $ l_i$, and further, the link given by $(lk_{F_1}(x), \ldots, lk_{F_m}(x) )$ is isotopic to $L$. 

If each $l_i$ is an unknot, then there is a homeomorphism of pairs $$(\Dinf \tilde M, \amalg \Dinf F_i ) \cong (lk_{\tilde M} (x) , \amalg lk _{F_i} (x)  ).$$ In particular, $(\Dinf F_1 , \ldots, \Dinf F_m )$ is isotopic to $L$.
\end{cor}

\section{New obstruction}
\label{new obstruction}

Theorem \ref{localglobal} shows that an unknot in $S^3$ does not give an obstruction to Riemannian smoothing of $M$. Hence we now consider links with more than one components and produce a new obstruction which does not depend on the knottedness of individual components. We first recall a few properties of links. 

Let $l_1$ and $l_2$ be disjoint oriented curves in $S^3$. The \textit{linking number}, $Lk(l_1, l_2)$, can be defined to be the oriented intersection number of $l_2$ with a smooth bounding disc for the curve $l_1$. This linking number has the property that it if $l_1 \sim l_1'$ are homotopic to each other in the complement of $l_2$, then $Lk(l_1, l_2)=Lk(l_1', l_2)$. In particular, the linking number is invariant under an orientation-preserving homeomorphism $\phi : S^3 \rightarrow S^3$. In fact, it is also invariant under orientation-preserving near-homeomorphisms 
(\cite{DJL}, Appendix).
\begin{prop} 
\label{linking number under near homeo}
Let $\phi : S^3 \rightarrow S^3$ be an orientation-preserving near-homeomorphism
. Let $l_1, l_2$ be a pair of disjoint oriented curves in $S^3$, such that $\phi \circ l_1$ and $\phi \circ l_2$ are also disjoint curves in $S^3$. Then $Lk(l_1, l_2) = Lk(\phi \circ l_1, \phi \circ l_2)$.  
\end{prop} 

Let $L = (l_1, \ldots, l_n)$ be a non-trivial link in $S^3$ with $n >1$ components. 
We refer to the collection $\{ Lk(l_i, l_j) |\  i \not= j \}$ of linking numbers 
as the \textit{pairwise linking numbers of $L$}, or more simply, the linking numbers of $L$.

A link $L$ in $S^3$  is called a \textit{great circle link} if every component of $L$ is a geodesic in the standard metric on $S^3$. Recall that, a geodesic $g$ in $S^3$ is the intersection of a 2-plane in $\R^4$ with the unit sphere, that is a great circle in $S^3$. 
It is known that if a link $L$ in $S^3$ is a great circle link then the pairwise linking numbers of $L$ are $\pm 1$ \cite{W}.

\vspace{3mm}
Let $L = (l_1, \ldots , l_n)$ be an $n$-component link in $S^3$. Let $\Sigma$ be a triangulation of $S^3$ of type $L$ guaranteed by Theorem \ref{Sigma}. 
We show that $M^4 \defeq P_\Sigma$ will be the desired manifold with locally CAT(0) metric that does not support any Riemannian metric with non-positive sectional curvature.

By definition of $\Sigma$, every component $l_i$ of the link $L$ is given by a square in the triangulation. The complex $P_\Box$ corresponding to this square is isometric to a flat torus, say $T_i^2$. Let $v$ be a vertex of the complex $P_\Sigma$. The link of $v$ in $T^2_i$, $lk_{T^2_i}(v)$, is simplicially isomorphic to the $i$-th component $l_i$. Thus, $\coprod lk_{T^2_i}(v)$ is isomorphic to the link $L$ in $lk_M(v) \cong S^3$.  
 
Lifting to the universal cover, we obtain 2-dimensional flats $F_i \hookrightarrow \tilde M$ corresponding to each flat torus $T^2_i$. Let $\tilde v$ be a lift of $v$ in $\tilde M$. Then $ \coprod lk_{F_i}(\tilde v) \hookrightarrow lk_{\tilde M}(\tilde v)$ gives a copy of the link $L$ in $S^3$. We have that $\partial ^\infty F_i \cong S^1$ and $\partial ^\infty \tilde M \cong S^3$. Thus, in the boundary, the embedding $\coprod \partial ^\infty F_i \hookrightarrow \partial ^\infty \tilde M$ forms a link, say, $L_\infty$ in $\partial ^\infty \tilde M$.  
 

The geodesic retraction $\rho : \partial ^\infty \tilde M \rightarrow lk_{\tilde M}(\tilde v)$ maps the link $L_\infty$ to $L$. We know that $\rho$ is a near-homeomorphism \cite{Ar}, and hence the pairwise linking numbers of $L_\infty$ are same as that of $L$ (by Prop. \ref{linking number under near homeo}).

Let $\Gamma = \pi _1 (M)$. Assume that there is a non-positively curved Riemannian manifold $M'$ with $\pi_1(M') \cong \Gamma$. As shown in Section 2, the isolated squares condition in $\Sigma$ ensures that the isolated flats condition holds in $\tilde M$. This is turn gives the existence of a $\Gamma$-equivariant homeomorphism $\phi: \partial ^\infty \tilde M \rightarrow \partial^\infty \tilde M'$. 

There exists a copy of $\Z^2$ in $\Gamma \cong \pi_1(M')$ such that the limit set of $\Z ^2$ is $\partial ^\infty F_i$. By the flat torus theorem, for each copy of $\Z ^2$ there exists a $\Z^2$-periodic flat $F_i' \subset \tilde M'$, with the property that $\partial ^\infty F_i'$ coincides with the limit set of $\Z^2$. The $\Gamma$-equivariant homeomorphism $\phi$ then maps $\partial ^\infty F_i$ to $\partial ^\infty F_i'$.
Thus we have a homeomorphism $\partial ^\infty \tilde M \rightarrow \partial ^\infty \tilde M'$ which maps the link $L_\infty$ to a link $L_\infty '$ in $\partial ^\infty \tilde M' \cong S^3$. 

Let us now focus on links with 2 components. Let $L=(l_1, l_2)$. By Prop. \ref{linking number under near homeo}, the linking number between $l_1$ and $l_2$ is preserved under the near-homeomorphism $\rho$. It is also preserved under the homeomorphism $\phi$.


If the linking number, $Lk(l_1, l_2)$, is non-zero, then the flats must intersect in a unique point, say $p$, in $\tilde M'$. Let us further assume that the linking number is not $\pm 1$. For example, consider the link in Figure \ref{lktwo} which has linking number 2.
 
\begin{wrapfigure}{r}{0.45\textwidth}
\centering
\includegraphics[scale=0.8]{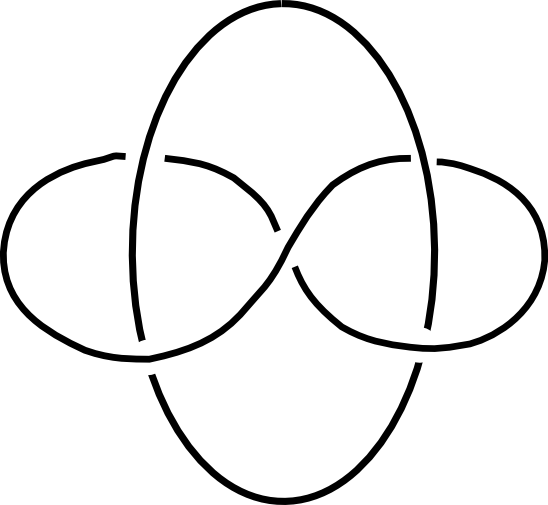}
\caption{Linking number = 2}
\vspace{-4mm}
\label{lktwo}
\end{wrapfigure}
The geodesic retraction map $\rho'$ from $\partial ^\infty \tilde M'$ to the unit tangent space $T_p\tilde M'$ maps $L_\infty '$ homeomorphically to some link $L'$ in $T_p \tilde M'$. A component $l_i'$ of this link is given by, $T_pF_i'$, the unit tangent space of $F_i'$ at $p$. Since the linking number is preserved under the homeomorphism $\rho'$, the link $(l_1', l_2') $ has linking number =2. However, $\coprod T_p F_i' \hookrightarrow T_p \tilde M'$ is a great circle link and, as noted earlier, great circle links have linking numbers $=\pm 1$. Thus, this gives a contradiction to the existence of the Riemannian manifold $M'$. 

\vspace{3mm}
\begin{wrapfigure}{r}{0.45\textwidth}
\begin{center}
\includegraphics[scale=0.8]{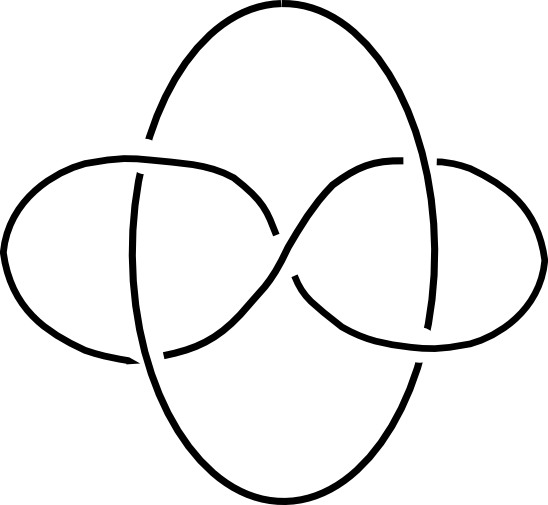}
\caption{Whitehead link, \\   Linking number = 0}
\label{Whitehead}
\end{center}
\end{wrapfigure}
Now, if the linking number is zero, then the flats will not intersect. If they intersected, then by the above argument we will obtain a link $L'$ in the unit tangent space $T_p\tilde M$ that is isotopic to $L'_\infty$. Hence, the linking number of $L'$ will be the same as $L$, zero. However, being a great circle link, the linking number of $L'$ must be $=\pm 1$. 

Let $p$ be a point in $\tilde M'$ which does not lie on either of the flats and is equidistant from both the flats. Let $D$ be the distance of each flat from $p$. Observe that for $r < D$, $S_p(r)$ does not intersect either of the flats. Further, there exists $r_0 > D$ such that the circles $S_p(r_0) \cap F_i$ form an unlink. 

For $r > D$, let $L'_r \defeq ( S_p(r) \cap F_1, S_p(r)\cap F_2)$ be a link in $S_p(r)$. Then $L'_{r_0}$ is an unlink. Since for $r> r_0$, $\rho^r_{r_0} : (S_p(r) , L'_r) \rightarrow (S_p(r_0), L'_{r_0})$ is a homeomorphism of pairs, every $L'_r$ ($r > r_0$) must be an unlink. 

By Corollary \ref{cor},  $(\Dinf \tilde M, L'_\infty) = \underset{\longleftarrow}{\lim}\ ((S_p(r) , L'_r), \rho_r)$ must be homeomorphic to $(S^3, U)$ where $U $ be an unlink with two components. But we know that $L'_\infty$ is a non-trivial link. This gives a contradiction. 



An example of a non-trivial 2-component link with linking number zero is given by the Whitehead link as shown in Figure \ref{Whitehead}. 

Thus 2-component links with linking number $\not= \pm 1$ give us examples of 4-manifolds with locally CAT(0) metric, but no Riemannian metric with non-positive sectional curvature.  

\vspace{3mm}
Now suppose $L = (l_1, \ldots , l_n)$ is a link with more than 2 components. If one of the linking numbers, say $Lk(l_i, l_j)$, is not equal to $\pm 1$, then we have a 2-component link $(l_i, l_j)$ which satisfies the conditions of one of the cases above. Arguments same as above lead to a contradiction, showing that the manifold corresponding to $L$ cannot support a non-positively curved Riemannian metric.


\section{Examples}

In this section we look at some specific examples of links that will give us the required obstruction.  \\

\noindent (1) Links with some pairwise linking number $\ge 2$. 

\begin{figure}[h]
\begin{subfigure}[b]{0.3\textwidth}
\centering
\includegraphics[scale=0.8]{lk2}
\vspace{4mm}
\caption{Link with 2 components with linking number = 2}
\end{subfigure}
\qquad\qquad\quad\qquad
\begin{subfigure}[b]{0.3\textwidth}
\centering
\includegraphics[scale=0.35]{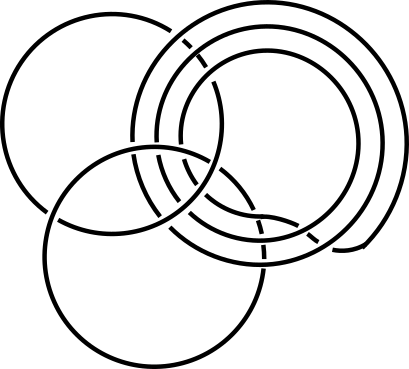}
\vspace{2mm}
\caption{Link with 3 components with linking numbers = 1, 3, 3}
\end{subfigure}
\end{figure}



\noindent (2) Links with some pairwise linking number $= 0$.   \\

\noindent (a) Links whose pairwise linking numbers belong to $\{ 0, 1 \}$.

\begin{wrapfigure}{r}{0.45\textwidth}
\vspace{-5mm}
\centering
  \includegraphics[scale=0.75]{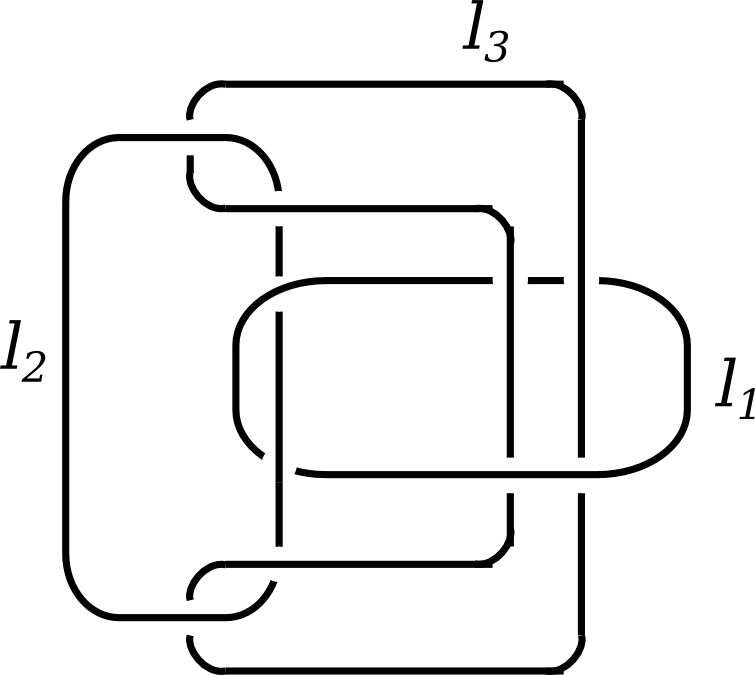}
  \caption{}
\label{lk=0,1}
  \end{wrapfigure}
Consider the link $L = (l_1, l_2, l_3) $ from Figure \ref{lk=0,1}. It has linking numbers  $Lk(l_1, l_2) = 1$ and $Lk (l_i, l_3) = 0$ for $i =1,2$.

As discussed before we have the geodesic retraction $\rho: \Dinf \tilde M \rightarrow lk_{\tilde M}(\tilde v)$ such that $\rho (L_\infty) = L$, where $\tilde v $ is a vertex in $\tilde M$ and $L_\infty $ is the link $\coprod \Dinf F$ in $\Dinf \tilde M$.

If there is a Riemannian manifold $M'$ with $\Gamma = \pi_1(M) \cong \pi_1(M')$, then as before, there is a $\Gamma$-equivariant homeomorphism $\phi: \partial ^\infty \tilde M \rightarrow \partial^\infty \tilde M'$ which maps $L_\infty$ to another link, say $L'_\infty \subset \Dinf \tilde M'$. The linking numbers are preserved under the near-homeomorphism $\rho$ (by Prop. \ref{linking number under near homeo}) and the homeomorphism $\phi$.

Since $Lk(l_1, l_2) = 1$, the corresponding flats $F_1'$ and $F_2'$ intersect in a point, say $p$. Let the corresponding components of $L'_\infty$ be $l'_{1,\infty}, l'_{2,\infty}$ and $l'_{3,\infty}$. Then the geodesic retraction map $\rho ' : \partial ^\infty \tilde M' \rightarrow T_p\tilde M'$ maps $l'_{1,\infty}, l'_{2,\infty}$ to the components $l_1'$ and $l_2'$. The link $(l_1', l_2')$ is a great circle link and hence has linking number =1. On the other hand, since $Lk (l_i, l_3) = 0$ ($i = 1, 2$), the flat $F_3'$ will not intersect with either of the flats $F_1'$ and $F_2'$. Let $r_0$ be the distance between $p$ and the flat $F_3'$. Then there exists $r>r_0$ such that $S_p(r) \cap (\coprod F_i' )$ is the link $L'_r$ as shown in Figure \ref{Ex3}. Then the geodesic retraction will map $L_\infty '$ to $L'_r$ in $S_p(r)$ homeomorphically, which is a contradiction for the following reason. 

 \begin{wrapfigure}{r}{0.47\textwidth}
\vspace{-6mm}
\centering
  \includegraphics[scale=0.85]{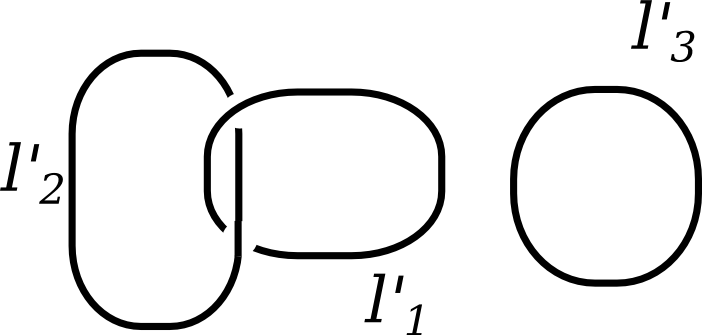}
  \caption{}
  \label{Ex3}
\vspace{-1mm}
  \end{wrapfigure}
The given link $L$ has each component homeomorphic to an unknot and $\rho$ is a near-homeomorphism. Hence, by Corollary \ref{cor}, we have that $L_\infty$ is isotopic to $L$. Since $\phi$ is a homeomorphism, we have that $L'_\infty$ is also isotopic to $L$. 
However, we know that the two links are not isotopic, showing that $M$ cannot support smooth Riemannian metric.\\ 


\noindent (b) Links with all pairwise linking numbers = 0

Besides the Whitehead link (Figure \ref{Whitehead}), examples of such links are given by Brunnian links. An $n$-component link is said to be a \textit{Brunnian link} if it becomes an unlink after removing any one of its components. Thus, every pair of components forms an unlink and hence the pairwise linking numbers of a Brunnian link are zero. The \textit{Borromean rings} are an example of a Brunnian link with 3 components (Figure \ref{Borromean}). 
An example of a family of Brunnian links is given by Milnor \cite{Mil1} as shown in Figure \ref{Brunnian}.

\begin{figure}[h]
\vspace{3mm}
\begin{subfigure}[b]{0.3\textwidth}
	\centering
	\includegraphics[width = 0.65\linewidth]{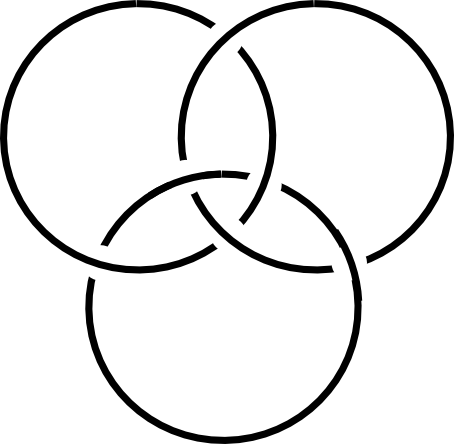}
\vspace{4mm}
	\caption{Borromean rings}
	\label{Borromean}
\end{subfigure}\quad
\begin{subfigure}[b]{0.65\textwidth}
	\centering
	\includegraphics[width = 1\linewidth]{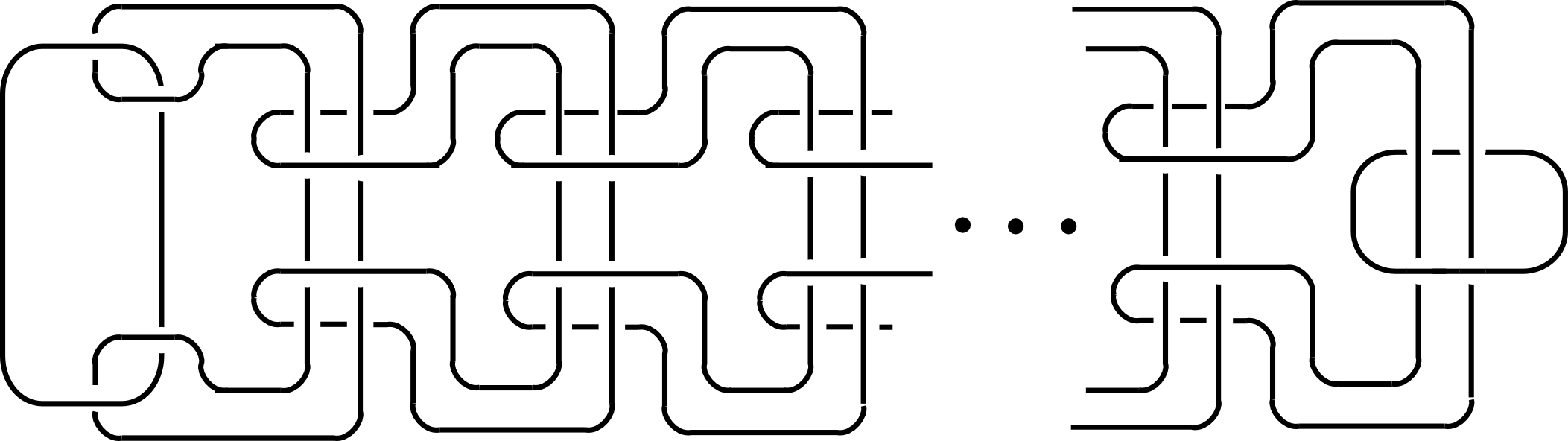}
\vspace{2mm}
	\caption{Brunnian link with $m$ components, $B_m$  ($m \ge 3$)}
	\label{Brunnian}
	\end{subfigure}
	\caption{Non-trivial links with pairwise linking numbers $= 0$}
	\label{Lk0}
\end{figure}

More examples of  nontrivial links with all the linking numbers $=0$ are given by the Whitehead doubles of the Brunnian links \cite{S}. An example of a Whitehead double with 3 components is given in Figure \ref{Whitehead double}. The Whitehead doubles with even number of twists have the additional property that they are link homotopically trivial. 


  
 \begin{figure}[h]
  \includegraphics[scale=0.7]{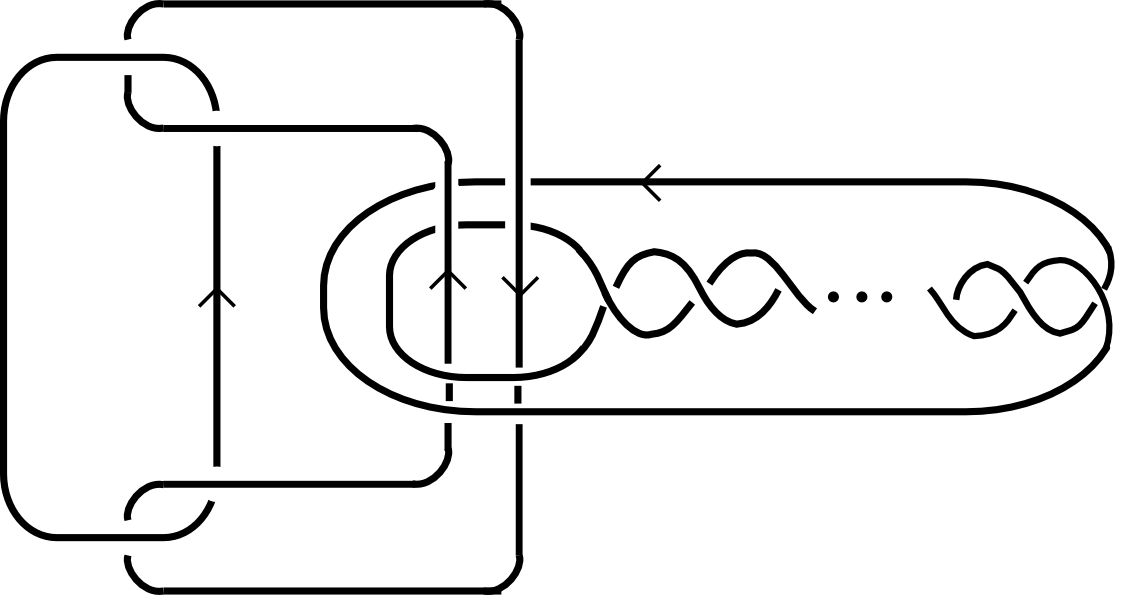}
  \caption{Whitehead double of Brunnian link $B_3$ with $n$ ($n$ even) twists}
\label{Whitehead double}
  \end{figure}


\section{Concluding Remarks}

\subsection{Linking number = 1} 
In this article we have developed examples of manifolds using links with linking numbers = 0 or $\ge 2$. This is because great circle links have linking numbers = 1 and great circle links do not obstruct Riemannian smoothings. However, if there exist links with linking  numbers = 1 that are not isotopic to any great circle link, then one would like to use them to produce examples of manifolds with the desired properties. The existence of such links can be shown by using results in \cite{S}. Theorem 1 in \cite{S} states that for a fixed $m$, there exist infinitely many $m$-component links with all linking numbers = 1. But since there are only finitely many great circle links with $m$ components, there must be infinitely many links with linking numbers = 1 that are not isotopic to any great cirlce link. These links also have the property that every proper sublink is isotopic to a great circle link.

Let $L$ be such a link, and let $P_\Sigma$ be the triangulation of $S^3$ of type $L$. The links in \cite{S} further have the property that each component is unknotted. Then Corollary \ref{cor} 
lets us conclude that $L_\infty \hookrightarrow \Dinf \tilde M$ is isotopic to $L$. Hence all linking numbers of $L_\infty$ are =1, but $L_\infty$ is not a great circle link. Now, if there is a Riemannian manifold with non-positive sectional curvature, $M'$, such that $\pi_1 (M) \cong \pi_1(M')$, then we get a $\pi_1(M)$-equivariant homeomorphism $\phi : \Dinf \tilde M \rightarrow \Dinf \tilde M'$ that carries $\Dinf F_i$ to $\Dinf F_i'$ where $F_i'$ are flats in $\tilde M'$. 
 If all the flats $F_i'$ intersect in a single point, then arguments as in Section 4 follow to show that there is a link in $T_p\tilde M'$ for some $p \in \tilde M'$ which is a not a great circle link, and we get a contradiction. 
 We know that each pair of flats $F_i', F_j'$ will intersect in a unique point in $\tilde M'$. However, there is no way to know if all the flats $F_i'$ will intersect in a single point. 
  
  \vskip 10pt
\noindent \textbf{Question:} Does a link $L$ as above 
provide an obstruction to Riemannian smoothing of locally CAT(0) 4-manifolds?

  \vskip 10pt
Note that if $L$ does provide an obstruction, then it would yield an example of a CAT(0) 4-manifold where no single flat or a pair of flats obstruct the smoothing, but where a triple of flats obstructs it.

\subsection{Further Questions}

Here we point out a few interesting questions that arise from this work. The examples that are given in this article and in \cite{DJL} are locally CAT(0) 4-manifolds whose universal covers are diffeomorphic to $\R^4$. In contrast, the examples given by Davis-Januszkiewicz \cite{DJ} for dimensions $\ge 5$ are closed aspherical manifolds whose universal covers are not homeomorphic to $\R^n$. So we reiterate the question posed by Davis-Januszkiewicz-Lafont \cite{DJL}.\\
Can one find locally CAT(0) closed 4-manifolds $M^4$ with the property that their universal covers $\tilde M^4$ are
\begin{enumerate}
\item not homeomorphic to $\R^4$?
\item homeomorphic, but not diffeomorphic to $\R^4$?
\end{enumerate}

Another interesting question is whether one can produce similar examples in the CAT(-1) setting. The construction presented here and in \cite{DJL} depends on the existence of flats. If one doesn't have flats, we have the following question. \\
Can one construct examples of smooth, locally CAT($-1$) manifolds $M^n$ with the property that $\Dinf \tilde M$ is homeomorphic to $S^{n-1}$, but which do not support any Riemannian metric of non-positive sectional curvature?

\end{document}